\documentclass[reqno]{amsart}
\usepackage{hyperref}
\usepackage{graphicx}
\pdfoutput=1
\begin{document}

\title {Uniqueness and radial symmetry of minimizers for a nonlocal variational problem}

\author{ Orlando Lopes}

\address{Orlando Lopes \newline
IMEUSP- Rua do Matao, 1010, CEP: 05508-090, Sao Paulo, SP, Brazil}
\email{olopes.usp@gmail.com}



\begin{abstract}

For $-n<p<0$ and $ 0<q$ and
$$K(x)=\frac{\|x\|^q}{q} - \frac{\|x\|^p}{p}$$
  the existence of minimizers of
$$E(u)=\int_{R^n\times R^n} K(x-y) u(x)u(y) \,dx \,dy$$
under
$$ \int_{R^n}u(x) \, dx=m>0; \quad 0 \leq  u(x) \leq M$$
 with given $m$ and $M$ has been proved. In this  paper we show that for $2\leq q \leq 4$  and
  except for translation, minimizer is unique and hence, radially symmetric. Applications are given for powers $(p,q)$ equal to $(-1,2)$, $(-1,3)$ and $(-1,4).$ As we will see, the shape of the minimizer depends on the ratio $m/M.$
 \end{abstract}

 \maketitle
\numberwithin{equation}{section}
\newtheorem{theorem}{Theorem}[section]
\newtheorem{lemma}[theorem]{Lemma}
\newtheorem{definition}[theorem]{Definition}
\newtheorem{remark}[theorem]{Remark}
\allowdisplaybreaks

\section{Introduction and Statement of the Result}

Functionals of the type $$E(u)=\int_{R^n\times R^n} K(x-y) u(x)u(y) \,dx \,dy$$
where $K(x)$ is given above,
are connected with the modelling of several phenomena such as self-assembly/aggregation models
 (\cite{choksi1} and \cite{fetecau}) and flocking of birds and  some other condensation phenomenon (\cite{burchard}).

If $M>0$ and $m>0$ are given and $\|x\|$ denotes the euclidean norm in $R^n$,  following
 \cite{choksi1} we define the set
\begin{equation}
{\mathcal A}=\{u\in L^1(R^n) \cap L^{\infty}(R^n): u\geq 0, \quad \|u\|_{\infty}\leq M \quad \rm{and} \quad \int_{R^n} u(x) \, dx=m\}
\end{equation}
and the functional
\begin{eqnarray}
&&E(u)=F(u)+G(u)\nonumber \\
&=&\int_{R^n\times R^n} \frac{\|x-y\|^q}{q}  u(x)u(y) \,dx \,dy-\int_{R^n\times R^n}\frac{ \|x-y\|^p}{p} u(x)u(y) \,dx \,dy.
\end{eqnarray}
For $-n<p<0$ and $ 0<q$ the existence of minimizers of $E$ on ${\mathcal A}$ has been proved in
 \cite{choksi1}.
So far, uniqueness and radial symmetry of the minimizer has been proved for  $q=2$ only. The uniqueness is a consequence of the convexity of $E(u)$ on the admissible set.

In \cite{frank1} the existence of certain classes of solutions is proved for $p=1$ and $n=3.$

For $n=3$, $p=1$ and $q\geq 2$, it has been proved that the radially symmetric equilibrium is unique and compactly  supported (see \cite{fetecau}).

Our main result is the following:
\begin{theorem} If $2\leq q \leq 4$ and $-n<p<0$ then, except for translation,  minimizer is unique. In particular, it is radially symmetric.
 \end{theorem}
 The uniqueness will be a consequence of the fact that  $E(u)$ is convex  on the admissible set. The convexity of the second functional that appears in the definition of $E(u)$ is very well
 known. The proof of the convexity of the first  is the main contribution of
 this paper (theorem 2.4).

If the $L^{\infty}$ condition is removed from the definition of ${\mathcal A}$, in \cite{choksi1}
the existence of minimizer for
$$E(\mu)=\int_{R^n\times R^n} K(x-y) d\mu (x)d\mu (y) $$
is proved in the set of the probability measures $\mu.$
It would be interesting to know if  our technique can be extend to prove the uniqueness of
the minimizing measure.

In section 3 and for $n=3$,  we give examples of minimizers for the powers $p=-1,q=2,$ $p=-1,q=3$ and $p=-1,q=4.$ Perhaps the most interesting case is $p=-1,q=4$ for which we construct (with computer assistance) radially symmetric minimizers $u(r)$  such that both sets $\{r: 0<u(r)<M\}$ and $\{r: u(r)=M\}$ have positive measure.

 \section {Proof of the main result}

The proof will be broken is several lemmas. The proof of the first is elementary.
\begin{lemma}  If, as before,  $\|x\|$ denotes the euclidean norm in $R^n$ and $q\geq 1$ we have
\begin{equation}
\|x-y\|^q \geq \left(\frac{1}{2}\right)^{(q-1)}\|x\|^q-\|y\|^q
\end{equation}

\begin{equation}
\|x-y\|^q \leq 2^{q-1}(\|x\|^q+\|y\|^q)
\end{equation}

\begin{equation}
\|x+y\|^q \leq 2^{q-1}(\|x\|^q+\|y\|^q)
\end{equation}

\end{lemma}
\begin{remark}  Instead of (2.1), all we need is $\|x-y\|^q \geq c_1\|x\|^q-c_2\|y\|^q$ and this can be achieved taking $0<c_1<1$ and $c_2$ convenient.
\end{remark}

\begin{lemma}  If $u\in {\mathcal A}$ then
$$ \int_{R^n\times R^n} \|x-y\|^q u(x)u(y) \,dx \,dy$$ is finite if and only if
$$\int_{R^n} \|x\|^q u(x) \, dx$$
is finite. In that case, the integrals $$  \int_{R^n} \|x\|^r u(x) \, dx$$ are also finite for $1\leq r \leq q.$

\end{lemma}
\begin{proof}

If $R$ is such that
$$\int_{\|y\|\leq R} u(y)\,dy= m/2,$$
using (2.1) we get $$ \int_{R^n\times R^n} \|x-y\|^q u(x)u(y) \,dx \,dy \geq \int_{ \|y\|\leq R} (\int_{R^n}
  \|x-y\|^q u(x)u(y) \,dx) \,dy\geq $$
$$\left(\frac{1}{2}\right)^{(q-1)} \int_{ \|y\|\leq R} (\int_{R^n}  \|x\|^q u(x)u(y) \,dx) \,dy- \int_{ \|y\|\leq R} (\int_{R^n}  \|y\|^q u(x)u(y) \,dx) \,dy$$
$$ \geq \frac{m}{2} \left(\frac{1}{2}\right)^{(q-1)} \int_{R^n}  \|x\|^q u(x) \,dx-m M \int_{ \|y\|\leq R}   \|y\|^q  \,dy.$$

We  have used  $u(y)\leq M.$ Therefore,$\displaystyle \int_{R^n}  \|x\|^q u(x) \,dx $ is finite and then
$$\int_{R^n} \|x\|^r u(x) \,dx $$ is also finite for $1\leq r \leq q$   because

$$\int_{R^n} \|x\|^r u(x) \, dx=\int_{\|x\|\leq 1} + \int_{\|x\|> 1} \leq m +\int_{\|x\|> 1}\|x\|^q u(x) \,dx .$$

Conversely, if $$\int_{R^n} \|x\|^q u(x) \, dx$$
is finite, using (2.3) we see that  the integral
$$ \int_{R^n\times R^n} \|x-y\|^q u(x)u(y) \,dx \,dy \leq 2^{q-1} \int_{R^n\times R^n} (\|x\|^q+\|y\|^q) u(x)u(y) \,dx \,dy$$
$$= 2^{q-1} 2m \int_{R^n} \|x\|^q u(x) \, dx$$
is also finite and the lemma is proved.
\end{proof}

In view of the previous lemma, we  define the Banach space
\begin{equation}X=\{h\in L_1(R^n) : \int_{R^n} ( \|x\|^q+1) |h(x)| \, dx < \infty
\end{equation}
and we redefine
\begin{equation}
{\mathcal A}=\{u\in X : u\geq 0, \quad \|u\|_{\infty}\leq M \quad \rm{and} \quad \int_{R^n} u(x) \, dx=m\}
\end{equation}

We also define
\begin{equation}
X_0=\{h\in X :  \int_{R^n} h(x) \,dx=0;
\int_{R^n} x_ih(x) \,dx=0, \quad 1\leq i \leq n\}
\end{equation}
and   we  consider the quadratic form $F: X \rightarrow R$
\begin{equation}
F(h)= \int_{R^n\times R^n} \|x-y\|^q h(x)h(y) \,dx \,dy.
\end{equation}
Clearly $F(h)$ is continuous. Our main result is the following:
\begin{theorem} For $h\in X_0 $   and $2\leq q \leq 4$, we have
 \begin{equation}
F(h)\geq  0.
 \end{equation}
\end{theorem}
 \begin{proof}
 If $q=2$, following \cite{choksi1}, we have

 $$F(h)= \int_{R^n\times R^n} \|x-y\|^2 h(x)h(y) \,dx \,dy$$
 $$=\int_{R^n\times R^n}( \|x\|^2-2\langle x,y\rangle+\|y\|^2) h(x)h(y) \,dx \,dy=0.$$

 If  $q=4$ we have
$$F(h)= \int_{R^n\times R^n} \|x-y\|^4 h(x)h(y) \,dx \,dy=$$
  $$\int_{R^n\times R^n} (\langle x,x\rangle -2 \langle x,y\rangle+\langle y,y\rangle)^2  h(x)h(y) \,dx \,dy.$$
Expanding the square  and using the definition of (2.6) of the space  $X_0$  we get
\begin{equation}
F(h)=2\left (\int_{R^n}\langle x,x\rangle h(x) \,dx \right)^2 +4\int_{R^n} (\langle x,y\rangle)^2 h(x)h(y) \,dx \,dy.
\end{equation}

Moreover $\langle x,y\rangle^2$ is a sum with positive coefficients of $x_i^2y_i^2$ and of $x_ix_jy_iy_j$ with $i\neq j.$ Therefore the second term of $F(h)$ in (2.9) is the sum of
$$ \left (\int_{R^n} x_i^2 h(x) \,dx \right)^2 \quad \rm{and} \quad
 \left(\int_{R^n} x_ix_j h(x) \,dx \right)^2.$$
and then  $F(h) \geq 0$ for $q=4.$

For  $2<q<4$  we start with  $h \in X_0 \cap  {\mathcal{S}}(R^n)$, where ${\mathcal{S}}(R^n)$
is  the Schwartz space. If $\hat{h}(\xi)$ denotes the Fourier transform of $h(x)$,
from the definition (2.6)  of the space $X_0$ we have
\begin{equation}
\hat{h}(0)=0; \qquad \frac{ \partial \hat{h}(0)}{\partial \xi_i}=0;\ \ i=1, \cdots,n.
\end{equation}
Next we notice  that $\hat{h}(\xi)$ is a $C^2$ function with bounded second derivatives
because the integral
$$ \int_{R^n} ( 1+\|x\|^2) |h(x)| \,dx$$
is finite.

By Parseval and convolution we also have
\begin{equation}
F(h)= \int_{R^n}\widehat{ \|x\|^q}(\xi) |\hat{h}(\xi)|^2 \, d \xi.
\end{equation}
But
$$\widehat{ \|x\|^q}(\xi)=C(q)\|\xi\|^{-q-n}$$
where
$$C(q)=2^{q+n/2}\frac{ \Gamma((q+n)/2)}{\Gamma (-q/2)}.$$
For a proof see \cite{gelfand}, chapter II, section 3. Therefore, $C(q) >0$ for $2<q<4$ because $\Gamma(z)$ is positive for either  $z>0$ or $-2<z<-1.$ Notice that $C(q)$ has a singularity at $q=2$ and $q=4.$ That is why those cases have been treated separately.
Since $\|\xi\|^{-q-n}$ is not in $L^1_{loc}$ at $\xi=0$,  the right hand side of (2.11)
\begin{equation} F(h)= C(q) \int_{R^n}\|\xi\|^{-q-n} |\hat{h}(\xi)|^2   \, d \xi.
\end{equation}
has to be understood  in the sense of analytic continuation of a tempered distribution.
 However, in view of (2.10), we see that the derivatives of
 $|\hat{h}(\xi)|^2$ vanish at $\xi=0$ up to order three. In fact, the right hand side of (2.11) can be written as
 \begin{equation}
 F(h)= C(q) \int_{R^n}\|\xi\|^{-q-n} \|\xi\|^4 \left|\frac{\hat{h}(\xi)}{\|\xi\|^2} \right |^2 \, d \xi.
 \end{equation}

Now we see that   $\|\xi\|^{-q-n+4}$ does belong to $L^1_{loc}$ for $q<4$ and
$\displaystyle{\left|\frac{\hat{h}(\xi)}{\|\xi\|^2} \right |^2}$ is bounded  at $\xi=0$  in view of (2.10) and the fact that $\hat{h}(\xi)$ is  $C^2$.  We conclude that $F(h) \geq 0$
  for $h \in X_0 \cap  {\mathcal{S}}(R^n).$

 Since (2.13) makes sense for $h$ in the space $X_0$ defined by (2.6) , we expect it to hold in this larger set. Taking convenient approximations, we next  sketch a proof for  that statement.

First we assume that $h(x)=0$ for $\|x\| >R$ and we denote by $\rho:R^n \rightarrow R$  a nonnegative  $C^{\infty}$ function that vanishes for $\|x\|\geq 1$ and has integral equal to one in $R^n$. As usual,

$$\rho_{\epsilon}(x)=\epsilon ^{-n} \rho(x/\epsilon).$$

For $0<\epsilon \leq 1$  defining
$$h_{\epsilon}(x)=(\rho_{\epsilon}*h)(x)=\int_{R^n} \rho_{\epsilon}(x-y)h(y) \, dy,$$
 we have
\begin{equation}
 \int_{R^n} h_{\epsilon}(x) \,dx=0
 \end{equation}
 and
\begin{equation}
\int_{R^n} x_ih_{\epsilon} (x) \,dx=\int_{R^n\times R^n}(y_i+z_i) \rho_{\epsilon}(z)h(y) \, dy \,
dz=0.
\end{equation}
Another way to see that (2.14) and (2.15) hold is to look at
$ \hat{h}_{\epsilon}(\xi) =\hat{\rho}_ {\epsilon}(\xi) \hat{h}(\xi).$
Therefore defining
$$F_1(\hat{h})=C(q) \int_{R^n}\|\xi\|^{-q-n} |\hat{h}(\xi)|^2   \, d \xi,$$
since $h\in X_0 \cap {\mathcal{S}}(R^n)$, (actually $h\in X_0 \cap {\mathcal{D}}(R^n)$) in view of (2.14) and (2.15), we  have just proved that $F(h_{\epsilon})= F_1(\hat{h}_{\epsilon}).$
As $\epsilon$ tends to zero, we show that $F(h_{\epsilon})$ tends to $F(h)$ and $F_1(\hat{h_{\epsilon}})$ tends to $F_1(\hat{h}).$

Taking in account that
$$\int_{R^n}\|x\|^q |h_{\epsilon}(x)-h(x)| \,dx\leq    (R+1)^q \int_{\|x\|\leq R+1} |h_{\epsilon}(x)-h(x)| \,dx, $$
 and using  that $h_{\epsilon}$ tends to $h$ in $L^1(R^n)$, we conclude that
 $h_{\epsilon}$ tends to $h$ in the space $X$  and this shows $F(h_{\epsilon})$ converges to $F(h).$

To analyze the convergence of $ F_1(\hat{h_{\epsilon}})$ we write:

$$F_1(\hat{h_{\epsilon}})-F_1(\hat{h})= \int_{\|\xi\| \leq a}\|\xi\|^{-q-n+4} \left(\left|\frac{\hat{h}_{\epsilon}(\xi)}{\|\xi\|^2} \right |^2-\left|\frac{\hat{h}(\xi)}{\|\xi\|^2} \right |^2\right ) \, d \xi+$$
$$+ \int_{\|\xi\| \geq a}\|\xi\|^{-q-n} (|\hat{h_{\epsilon}}(\xi)|^2-|\hat{h}(\xi)|^2)  \, d \xi.$$

Using that  the second derivatives of $\hat{h_{\epsilon}}(\xi)$ are uniformly bounded,
given $\delta>0$ we first  choose $ a(\delta)$ in such way that
 the first integral is less than $  \delta/2.$ For that choice of $a$, the second integral can be made less than $\delta/2$ because $\hat{h_{\epsilon}}(\xi)$ converges to $\hat{h}(\xi)$
uniformly in $R^n$ and the integral

$$\int_{\|\xi\| \geq a}\|\xi\|^{-q-n} \, d\xi$$ is finite. This takes care of the case $h$ has compact support.

Next we take $ h\in X_0$ and we define the following functions:

$$g_0(x)=n/\omega_n   \quad \rm{for} \quad \|x\|\leq 1; \quad g_0(x)=0   \quad \rm{for} \quad\|x\|> 1.$$
$$g_i(x)=\frac{n(n+2)x_i}{\omega_n} \quad \rm{for} \quad \|x\|\leq 1; \quad g_i(x)=0 \quad \rm{for} \quad \|x\|> 1$$

Then
$$ \int_{R^n} g_0(x) \, dx =1; \quad \int_{R^n}x_i g_0(x) \, dx=0;  \quad i=1, \cdots,n $$
$$ \int_{R^n} g_i(x) \, dx =0; \quad \int_{R^n}x_j g_i(x) \, dx= \delta_{ij};  \quad   i=1, \cdots,n $$
If we define
$$ \phi_m(x)= \left( \int_{ \|x\| \leq m} h(x) \,dx \right)g_0(x)+ \sum_{i=1}^{n}\left (\int_{ \|x\| \leq m} x_ih(x) \,dx \right)g_i(x)$$
then the function
$$h_m(x)=h(x)-\phi_m(x) \quad \rm{for} \quad \|x\| \leq m; \quad  h_m(x)=0 \quad \|x\| \leq m $$ satisfies
$$ \int_{R^n} h_m(x) \, dx=0; \quad \int_{R^n} x_i h_m(x) \, dx=0; \quad i=1, \cdots,n $$
and then $F(h_m)=F_1( \hat{h_m}(\xi)).$ Moreover
$$ \int_{R^n} (\|x\|^q+1)|h_m(x)-h(x)| \, dx$$
goes to zero as $m$ tends to infinity and then arguing as before we conclude

$F(h)=F_1( \hat{h}(\xi))$  and the theorem is proved.
\end{proof}

For a proof of next lemma se \cite{lieb1}.
\begin{lemma} For $-n<p<0$ and  $h\in  L^1(R^n) \cap L^{\infty}(R^n)$ the quadratic form
$$G(h)=\int_{R^n\times R^n} \|x-y\|^p h(x)h(y) \,dx \,dy$$
is well defined and $G(h)>0$ if $h\not\equiv 0.$
 \end{lemma}
\noindent {\bf Proof of the theorem 1.1} Suppose $u$ and $v$ are distinct minimizers. Following
\cite {choksi1}, we can make  a translation in the space variable
(possibly different translations for $u$ and $v$) in such way that
$$\int_{R^n} x_iu(x) \,dx=0=\int_{R^n} x_iv(x) \,dx=0;  \quad i=1, \cdots ,n.$$
We  keep the same notation for the translated functions.
Therefore, the function $h=v-u$ belongs to the space $X_0$ defined by (2.6).

 Defining  $\phi(t)= E((1-t)u+tv)$ we have $\phi(0)=\phi(1)$, $\phi'(0)\geq 0$ (because
  $(1-t)u+tv$ is admissible for $0\leq t \leq 1$)  and $$\phi''(t)=2F(h)+2G(h)>0$$ (in view of theorem 2.4 and lemma 2.5).
 This is a  contradiction and  uniqueness is proved.

 To prove the radial symmetry, suppose  $u(x)$ is a minimizer satisfying
 $$\int_{R^n} x_iu(x) \,dx =0, i=1, \cdots ,n$$ This condition can be written in the vector form
    $$\int_{R^n} xu(x) \,dx =0.$$
     If  $C$ is any orthogonal matrix and  $v(x)=u(Cx)$, then $v$ is also
     a minimizer and
     $$ \int_{R^n} xv(x)=C^{-1}\int_{R^n} xu(x) \,dx =0$$
 Therefore  we must have $u(x)=v(x)=u(Cx)$ and this implies the radial symmetry of
 the minimizer and the theorem is proved.

 Next theorem will be useful for the applications we are going to make.

\begin{theorem}   Let $ u\in {\mathcal A}$ be a radially symmetric function and let $\phi(t)$ defined
 in the proof of the theorem1.1. If $0<p<n$ and $2\leq q \leq 4$,
then $u$ minimizes $E$ if and only if
\begin{equation}
\phi'(0)=E'(u)(v-u)=
2\int_{R^n\times R^n}K(x-y) u(x)(v(y)-u(y) \,dx \,dy\geq 0
\end{equation}
 for any radially symmetric function $v \in {\mathcal A}.$ Or, equivalently,
 \begin{equation}
\int_{R^n\times R^n}K(x-y) u(x)v(y) \,dx \,dy\geq \int_{R^n\times R^n}K(x-y) u(x)u(y) \,dx \,dy
\end{equation}
 for any radially symmetric function $v \in {\mathcal A}.$
\end{theorem}

 \begin{proof} The proof follows immediately from theorem 1.1.
 \end{proof}
\section{ Examples}

As before, for given $M>0$ and $m>0$ and $X$ defined by (2.4), we consider the set of admissible functions
\begin{equation}
{\mathcal A}=\{u\in X : u\geq 0, \quad \|u\|_{\infty}\leq M \quad \rm{and} \quad \int_{R^n} u(x) \, dx=m.\}
\end{equation}
 For $-n<p<0$ and $ 0<q$ we define
$$K(x)=\frac{\|x\|^q}{q} - \frac{\|x\|^p}{p}$$
and
$$E(u)=\int_{R^n\times R^n} K(x-y) u(x)u(y) \,dx \,dy$$
We define problem $P$ by
\begin{equation}
(P): \quad \mbox{Minimize}\quad E(u)\quad  \mbox{for} \quad u \in {\mathcal A}.
\end{equation}

In  \cite{frank1}   problem $P$ with $p=-1$, $q>0$ and $n=3$   is considered   and  three phases of the minimizers  are defined:

Phase 1: $|\{x: u(x)=M\}| = 0$,

Phase 2: $0<|(\{x: u(x)=M\}| <m/M$,

Phase 3:  $|(\{x: u(x)=M\}|=m/M,$

where $| \cdot|$ denotes the Lebesgue measure in $R^3.$ There it is shown that the minimizer is of phase 1 if the ratio $m/M$ is below a certain critical value and it is of
phase 3 if the ratio $m/M$ is above  a certain  (perhaps different) critical value. In the case $p=-1, q=2$ the minimizers are known for all values of the ratio $m/M$ and, as a consequence, it is known that that phase 2 does not occur. Therefore, in that case, as the ratio $m/M$ increases, the phase jumps directly from phase 1 to phase 3.

Here in this paper we adopt the same terminology and we construct very explicitly the phase 3 minimizers  also in the cases $p=-1, q=3$ and $p=-1, q=4$ and we exhibit the critical ratio.

As far as  phase 1 minimizers are concerned, in the case $p=-1, q=4$ they are constructed very explicitly and  the critical ratio is calculated. In the case $p=-1, q=3$ the construction is  less explicit because it depends on numerical calculations.

Finally, in the case $p=-1, q=4$  we construct (with computer assistance)  phase 2 minimizers.  There is a strong indication the Phase 2 occurs also in the case $p=-1, q=3$ but the  calculations are heavier.

Although the construction of the minimizers is carried out for particular powers, it may give a good insight for more general cases.

Let us emphasize that the fact that  the functions that we are going to construct are indeed minimizers is a consequence of our uniqueness result.

We have performed the calculation and the plotting using REDUCE.

We start with some  necessary conditions for minimizers for the problem $P$  in the general case (see \cite{choksi1}). For $-n<p<0$ and $ 0<q$  and for a given function  $u:R^n \rightarrow R$, we define the function
\begin{equation} \Lambda (x)= \int_{R^n}K(x-y) u(y) \, dy
\end{equation}
where
$$K(x)=\frac{\|x\|^q}{q} - \frac{\|x\|^p}{p}$$

If $u\in \mathcal{A}$ is a minimizer of problem $P$ and the set $\{x: 0<u(x)<m\}$ has positive measure, then
there is a $\eta>0$ such that

\begin{equation} \Lambda (x)= \left\{\begin{array}{ll}
\eta & \mbox{if  $0<u(x) <M$} \\
\leq \eta & \mbox{if  $u(x)=M$} \\
\geq \eta & \mbox{if  $u(x)=0.$}
\end{array} \right.
\end{equation}
The second condition in (3.4) is not given in \cite{choksi1})  but it can be proved by the same method. For the  examples we encounter in this paper,  we prove that, basically, such conditions are also  sufficient.

We start with some elementary calculation involving radial functions. If $u(x)$ is a radial function and
$$ \Lambda_q (x)=\frac{1}{q} \int_{R^3}\|x-y\|^qu(y) \, dy$$
then $\Lambda_q (x)$ is also radial. Therefore we can assume that   $x=(0,0,r).$
Taking spherical coordinates $(s,\theta,\phi)$  in $y$

$$ y_1=s \sin \phi \cos \theta;\quad   y_2=s \sin \phi \sin \theta; \quad  y_3=s \cos \phi$$
we have
$$\|x-y\|^2=s^2-2r s \cos \phi + r^2.$$
Defining $b=q/2$,  we have to calculate
$$  \int_0^{\pi} ( s^2-2r s \cos \phi + r^2)^{b} \sin \phi \, d \phi$$
$$= \frac{1}{(q+2)rs} [( s +r)^{q+2}   -   |s -r|^{q+2}]. $$
If we drop the factor $2 \pi$ corresponding to the integral of $d\theta$ we get

$$\Lambda_q (r)= \int_0^{\infty} \frac{1}{q(q+2)r} [( s+r)^{q+2}  - |s -r|^{q+2}]s u(s)\, ds$$
Using a similar formula for the second term with $p\neq -2$ (if $p=-2$ a logarithm arises) and defining the function
$$K(r,s)=  \frac{1}{q(q+2)r} [( s+r)^{q+2}  - |s -r|^{q+2}]  -  \frac{1}{p(p+2)r} [( s +r)^{p+2})
  - |s -r|^{p+2}]$$
we see that the function defined by (3.3) can be written as

\begin{equation}\Lambda (r)= \int_0^{\infty} K(r,s) s u(s) \,ds.
  \end{equation}
Sometimes it is more convenient to deal with the function $w(r)=ru(r)$ and then (3.5) becomes
\begin{equation}
\Lambda (r)= \int_0^{\infty} K(r,s) w(s) \,ds.
\end{equation}

In view of those formulas for $\Lambda(r)$,  theorem  2.6  can be reformulated in the following way:
\begin{theorem}
If $u$ is a radial function belonging to ${\mathcal A}$ then $u$ is a minimizer for problem $P$
if and only if
\begin{equation}
\int_0^{\infty} \Lambda(r) v(r) r\,dr \geq \int_0^{\infty} \Lambda(r) u(r) r\,dr
\end{equation}
for any radial function $v(r)$  such that $0\leq v(r) \leq M$ and
$$ \int_0^{\infty}r^2v(r) \, dr=\int_0^{\infty}r^2u(r) \, dr.$$
\end{theorem}
All proofs of the sufficient conditions we are going to give rely on theorem 3.1.

We start with the phase 3 minimizers, that is, minimizers that assume the values $M$ and zero only.

For  $M>0$ and $a>0$ we  define  a radially symmetric function $u(r)$ by
\begin{equation}
u(r)=M \quad \rm{for}\   0 \leq r \leq a ;\quad  \mbox{and} \quad u(r)=0 \quad \rm{for}\  \quad a<r,
\end{equation}
and
 $$m=\int_{R^3} u(x) \, dx= \frac{ 4\pi}{3} M a^3$$
 We will give necessary and sufficient conditions for $u$ to be a minimizer for particular exponents.
  To start with we give a necessary and sufficient conditions in terms of the function $\Lambda (r)$.

 \begin{theorem} Let $M,a>0$ and $u(r)$,  $\Lambda (r)$ and $m$ as above. Then $u$ is a minimizer for problem $P$  if and only if the following condition holds
 \begin{equation}
 \Lambda (r) \leq \Lambda (a)\quad   \mbox {for}\quad   0\leq r \leq a\quad  \mbox{and}\quad   \Lambda (a) \leq \Lambda (r) \quad \mbox {for} \quad  a \leq r.
 \end{equation}
\end{theorem}
\begin{proof}  Suppose the condition (3.9) holds. Then for any radially symmetric function $v(r)$ satisfying
$$ 0\leq v(r) \leq M \mbox{ and } \int_0^{\infty} v(r) r^2 \,dr=\int_0^{\infty} u(r) r^2 \,dr=M \int_0^a r^2\,dr$$
we have

$$\int_0^{\infty} \Lambda (r) r^2 v(r) \, dr-\int_0^{\infty} \Lambda (r) r^2 u(r) \, dr=$$

$$\int_0^{\infty} \Lambda (r) r^2 v(r) \, dr -M \int_0^R \Lambda(r)  r^2 \,dr   =$$
 $$\int_0^a \Lambda (r) v(r) r^2 \,dr + \int_a^{\infty}  \Lambda (r) v(r) r^2 \,dr-M \int_0^a \Lambda(r) r^2 \,dr \geq $$
 $$ \geq \int_0^a \Lambda (r) v(r) r^2 \,dr+ \Lambda (a)\int_a^{\infty} v(r) r^2 \,dr-M \int_0^R \Lambda(r) r^2 \,dr =$$
$$=\int_0^a \Lambda (r) v(r) r^2 \,dr+ \Lambda(a) (\int_0^{\infty} v(r) r^2 \,dr- \int_0^a   v(r) r^2 \,dr)-M \int_0^a \Lambda (r) r^2 \,dr= $$
$$\int_0^a \Lambda(r) v(r) r^2 \,dr-\Lambda(a)\int_0^a   v(r) r^2 \,dr +\Lambda(a) M \int_0^R r^2 \, dr -M \int_0^a \Lambda (r) r^2 \,dr= $$

$$\int_0^a (\Lambda (r)-\Lambda (a)) v(r) r^2 \,dr +M \int_0^a(\Lambda (a)- \Lambda(r)) r^2 \,dr\geq$$

$$M\int_0^a (\Lambda (r)-\Lambda (a))  r^2 \,dr  +M \int_0^a(\Lambda (a)- \Lambda(r)) r^2 \,dr=0.$$
Therefore, as a consequence of theorem 3.1, $u$ is a minimizer.
Conversely, suppose there is $r_0<R$ such that $\Lambda (r_0)>\Lambda(a)$. It is easy to see that there
are  real numbers
$r_1<r_2<a<r_3<r_4$ such that $\Lambda(r)>\Lambda(s)$ for $r_1<r<r_2$ and $r_3<s<r_4$ and the regions
$r_1\leq \|x\| \leq r_2$ $r_3\leq \|x\| \leq r_4$ have the same measure. If we define $v(r)$ by
$$v(r)=M\quad  \rm{for}\quad     0 \leq r \leq r_1\quad  \rm{or}\quad   r_2 \leq r \leq R\quad  \rm{or} \quad  r_3\leq r \leq r_4$$ and zero otherwise we
see that
$$\int_0^{\infty} \Lambda (r)v(r)r^2 \,dr < \int_0^{\infty} \Lambda(r)u(r)r^2 \,dr$$
 and, again as a consequence of theorem  3.1, $u$ is not a minimizer.

The other case is treated similarly and the theorem is proved.
\end{proof}
First we apply theorem 3.2 to the pair $ p=-1, q=2$. In that case,
$$K(r,s)=\frac{1}{r} \{\frac{1}{8}[(s+r)^4-(s-r)^4]+[s+r-|s-r|]\}$$
and
$$\Lambda(r)= M \int_0^a K(r,s)s \, ds.$$
Performing the calculation  and dropping the factor $ M$ we get

$$\Lambda (r) =(3a^5 + 5a^3r^2 + 15a^2 - 5r^2)/15 \quad \rm{for} \quad 0\leq r \leq a $$
$$\Lambda(r)= a^3( 3a^2r+ 5r^3+10)/(15r) \quad \rm{for} \quad a<r.$$
and then
$$ \Lambda(r)-\Lambda(a)= - (a^2 + a + 1)(a + r)(a - r)(a - 1)/3 \quad  \mbox {for}\quad   0\leq r \leq a$$
and
$$\Lambda(r)-\Lambda(a)= - (a^2r + ar^2 - 2)(a - r)a^2/(3r)\quad   \mbox {for}\quad   a\leq  r $$

 Elementary calculation shows that the conditions of theorem 3.2  are satisfies if and only if $a \geq 1.$ In terms of $M$ and $m$ this
is equivalent to $m/M \geq 4\pi/3.$ This agrees with \cite{choksi1}.

For $p=-1; q=3$   we have

$$K(r,s)=\frac{1}{r} \{\frac{1}{15}[(s+r)^5-|s-r|^5]+[s+r-|s-r|]\}$$

$$\Lambda(r)= (35a^6 + 105a^4r^2 + 21a^2r^4 + 315a^2 - r^6 - 105r^2)/315\quad \rm{for}
 \quad 0\leq r \leq a$$
and
$$\Lambda(r) = 2a^3(3a^4 + 42a^2r^2 + 35r^4 + 105)/(315r) \quad \mbox {for} \quad a<r. $$

Therefore,
$$ \Lambda(r)- \Lambda(a)= - (125a^4 + 20a^2r^2 - r^4 - 105)(a + r)(a - r)/315 \quad \mbox{for} \quad 0\leq r \leq a  $$
and
$$\Lambda(r)- \Lambda(a)=2(3a^4 - 77a^3r - 35a^2r^2 - 35ar^3 + 105)(a - r)a^2/(315r) \quad \mbox{for} \quad a<r $$

Working with these inequalities, we see that the conditions  of theorem 3.2  are satisfies if and only if $a^4  \geq 21/25.$ In terms of the ratio  $m/M$ that means
$$\frac{m}{M} \geq \frac{4\pi}{3} \left (\frac{21}{25}\right)^{3/4}=4\pi \times 0.292.$$

If $p=-1$ and $ q=4$ we have

$$K(r,s)=\frac{1}{r} \{\frac{1}{24}[(s+r)^6-(s-r)^6]+[s+r-|s-r|]\}$$

$$\Lambda(r)  =(3a^7 + 14a^5r^2 +7a^3r^4 +42a^2  - 14r^2 )/42 \quad \rm{for} \quad 0\leq r \leq a $$
and
$$\Lambda(r) =a^3( 3a^4r +14a^2r^3 + 7r^5  + 28)/(42r) \quad \rm{for} \quad  a<r  $$

Moreover

$$ \Lambda(r)- \Lambda(a)=  - (3a^5 + a^3r^2 - 2)(a + r)(a - r)/6
\quad \rm{for} \quad 0\leq r \leq a $$
and

$$ \Lambda(r)- \Lambda(a)=  - (3a^4r + 3a^3r^2 + a^2r^3 + ar^4 - 4)(a - r)a^2/(6r)
\quad \rm{for} \quad  a<r  $$

and the  conditions  of theorem 3.2  are satisfies if and only if $a^5  \geq 2/3.$ This is equivalent to say that
$$ \frac{m}{M} \geq \frac{4\pi}{3} \left(\frac{2}{3}\right)^{3/5}.$$
If we define a number $b_0$ by $b_0^5=2/3$, we see that $b_0$ is the radius of the phase 3 solution for the critical ratio and the critical ratio can be written as $\frac{4\pi}{3}b_0^3.$

Next we find phase 1 minimizers and we start with a sufficient condition.

\begin{theorem}  Let $u$  be a radially symmetric function  such that
$$ 0<u(r)<M \quad \, \mbox{for}\quad    0 <r<a\quad  \mbox{and}\quad  u=0 \quad \mbox{for} \quad  r\geq a$$
and for some $\eta>0$ we have
$$ \Lambda(r)=\eta \quad   \mbox {for} \quad 0\leq r \leq a \quad\mbox{and} \quad  \Lambda(r)\geq \eta \quad \mbox {for} \quad  a \leq r.$$
 Then $u$ is a minimizer.
\end{theorem}
\begin{proof} For any $v(r) $ satisfying $0 \leq v(r) \leq M$ and
$$ \int_0^{\infty} r^2v(r) \, dr= \int_0^{\infty} r^2u(r) \, dr$$
we have
$$ \int_0^{\infty} \Lambda(r) r^2v(r) \, dr -\int_0^{\infty} \Lambda(r)  r^2u(r) \, dr=$$
$$\int_0^a \Lambda(r) r^2v(r) \, dr  +\int_a^{\infty} \Lambda(r) r^2v(r) \, dr   -\int_0^a \Lambda(r)  r^2u(r) \, dr\geq $$

$$\eta  \int_0^a r^2v(r) \, dr  +\eta \int_a^{\infty} r^2v(r) \, dr   -\eta \int_0^a  r^2u(r) \, dr=0$$
and, as a consequence of theorem 3.1, the theorem is proved.

\end{proof}
In the cases we are going to consider, the phase 1 minimizers are of type given by the theorem 3.3.

We  take $p=-1$ and $q=2$, and suppose we want to find a minimizer $u(r)$  that fits in theorem 3.3. If, as before, we denote by $w(r)$ the function $ w(r)=ru(r)$, then the condition $ \Lambda(r)= \eta$ for
$0\leq r \leq a$ becomes:

 $$\int_0^a \{\frac{1}{8}[(s+r)^4-(s-r)^4]+[s+r-|s-r|]\}w(s)\,ds=\eta r$$

or
 \begin{equation}
 \int_0^r \{\frac{1}{8}[(s+r)^4-(s-r)^4]+2r\}w(s) \,ds
  +\int_r^a \{\frac{1}{8}[(s+r)^4-(s-r)^4]+[2s]\}w(s) \, ds =\eta r.
  \end{equation}
  If we define
  $$ z(r)=r\int_0^rw(s) \,ds + \int_r^a s w(s) \,ds$$
  and we assume that $w(s)$ is continuous in some interval, then $z''(r)=w(r)$. Having that in mind and  differentiating (3.10) two times with respect to $r$ get $w(r)=M_1r$, where $M_1$ is a constant.
  We conclude that  $u(r)=w(r)/r=M_1$ is constant. Next we show that such functions satisfy the conditions of theorem 3.3.

  In fact, if we define
  $$u(r)=M_1\quad  \mbox {for}\quad   0 \leq r \leq a\quad  \mbox{and}\quad  u(r)=0 \quad \mbox {for}\quad   a<r$$
  then for $0\leq r \leq a$ we have
  $$\Lambda(r)=\frac{1}{r} \int_0^a \{\frac{1}{8}[(s+r)^4-(s-r)^4]+[s+r-|s-r|]\}w(s) \,ds.$$
Therefore,  for  $0\leq r \leq a$, $\Lambda(r)$ is given by

$$\Lambda(r)= \frac{1}{r} \int_0^r \{\frac{1}{8}[(s+r)^4-(s-r)^4]+2s\}w(s) \,ds$$
$$+ \frac{1}{r} \int_r^a \{\frac{1}{8}[(s+r)^4-(s-r)^4]+2r\}w(s) \,ds$$
and
$$\Lambda(r)=\frac{1}{r} \int_0^a \{\frac{1}{8}[(s+r)^4-(s-r)^4]+2s\}w(s) \,ds\quad \mbox{for} \quad a<r.$$
Performing the calculation for $0\leq r \leq a$ we get
$$\Lambda(r)=(3a^5 + 5a^3r^2 + 15a^2 - 5r^2)M_1/15$$
and then $\Lambda(r)$ is constant if $a=1.$ In that case,
$$\Lambda(r)=M_1(5r^3 + 3r + 10)/(15r)\quad \mbox{for} \quad  1<r$$
and
$$ \Lambda(r)-\Lambda (1)= ((r + 2)(r - 1)^2M_1)/(3r)>0 \quad  \mbox {for}\quad   r>1$$
and this implies that $u(r)$ is the minimizer. As far the ratio $m/M$ is concerned we have
$$ \frac{m}{M}= \frac{1}{M_1}4\pi  \int_0^1 M_1 r^2 \, dr= \frac{4 \pi}{3}.$$
We conclude that if the ratio $m/M$ is less or equal to $\frac{4 \pi}{3}$, then $u(r)$ is the minimizer. Since there is no gap between the critical ratio for phase 1 and critical ratio for phase 3, we conclude, as it is well known (see \cite{burchard}), that there no phase 2 minimizers. However, as we will see, things are different if $q=3$ or $q=4.$

Now we construct phase 1 minimizers in the case $p=-1$ and $q=3$ and  we  impose

\begin{equation}
\int_0^a \{\frac{1}{15}[(s+r)^5-|s-r|^5]+[s+r-|s-r|]\}w(s) \, ds= \eta r
\quad \mbox{for} \quad   0\leq r \leq a.
\end{equation}
  That can be written as

\begin{eqnarray}
\int_0^r \{\frac{1}{15}[(s+r)^5-(r-s)^5]+2s]\}w(s) \, ds&+&\nonumber \\
\int_r^a \{\frac{1}{15}[(s+r)^5-(s-r)^5]+2r]\}w(s) \, ds&=& \eta r.
\end{eqnarray}
If $w(s)$ is continuous in some interval and we differentiate this last equation six times with respect to $r$ we get
$ w^{(4)}(r) + 8w(r)=0$ and  then
\begin{verbatim}
w(r)=c1* e**(p*r)*cos(p*r)+ c2*e**(-p*r)*cos(p*r)  + c3*e**(p*r)*sin(p*r)
+c4*e**(-p*r)*sin(p*r)
\end{verbatim}
with $p^4=2.$ If we replace this formula back  in (3.12), the left hand side is a polynomial of degree five in $r.$

Looking at the coefficients of $r^5,r^4$ and $r^2$ we get $ c_4=c_3; \quad c_2=-c_1$ and
\begin{verbatim}
c3:=(c1*(e**(2*a*p)*cos(a*p) + cos(a*p) +
e**(2*a*p)*sin(a*p) - sin(a*p)))/(e**(2*a*p)*cos(a*p) +
cos(a*p) - e**(2*a*p)*sin(a*p) + sin(a*p)).
\end{verbatim}

If we define $c=ap$, the coefficient of $r^3$ gives the following equations for $c$
\begin{verbatim}
 - 4*e**(2*c)*cos(c)**2 - 4*e**(2*c)*cos(c)*sin(c)*c + e**(4*c)*c -
e**(4*c) + 2*e**(2*c) - c - 1=0.
\end{verbatim}

Solving numerically we get $c\cong 1.09.$ Since  $p=2^{1/4}\cong 1.189$, we also have $a=c/p=0.916.$
Therefore we have
\begin{verbatim}
w(r)=c1* (e**(p*r)*cos(p*r)-e**(-p*r)*cos(p*r))  +
 c3*(e**(p*r)*sin(p*r)+e**(-p*r)*sin(p*r))
\end{verbatim}
and
\begin{verbatim}
w(r)/c1= e**(p*r)*cos(p*r)-e**(-p*r)*cos(p*r)  +
c5*(e**(p*r)*sin(p*r)+e**(-p*r)*sin(p*r))
\end{verbatim}
where $c_3$ is given in terms of $c_1$ by the formula above  and $c_5=c_3/c_1.$
Then, for $0\leq r \leq a$ the minimizer is given by  $u(r)=w(r)/r.$

If we plot $w_1=w(r)/c_1$ (see Plot1 ) we find that  it is negative and this means that $c_1$ has to be taken negative.
If we denote by $du$ the derivative of  $u(r)/c_1$ (see Plot2)  we find that  it is also negative. We conclude that the maximum of $u(r)$ is assumed at
$r=a$ and its value is $- 13.55c_1.$ To calculate the ratio $m/M$ we have to calculate the integral
$ 4\pi \int_0^a rw(r) \, dr=4\pi \times 0.22.$ Since the critical ratio for phase 3 minimizers is
$4\pi \times 0.292$, we see that there is a gap between those critical ratios.Probably this gap is filled by phase 2 minimizers. This question will be fully treated in the case $p=-1$ and $q=4.$ But, before that,
we have still  to verify that  $ \Lambda(r)- \eta= \Lambda (r) - \Lambda (a) \geq 0$ for $a\leq r$. After factorization and cancelation, that is equivalent to say that a certain polynomial
$p_1(r)= w_3r^3+ w_2r^2+w_1r+w_0 \geq 0$ for $r \geq a.$ The coefficients depend on $p$ and $c$ only and  their numerical value  can be  calculated explicitly. The answer is
$$w_0= -1877.39;\quad w_1=1769.62;\quad w_2=714.32; \quad w_3=654.29.$$  Since $w_1,w_2$ and $w_3$ are positive, the polynomial is monotonic in $r$ and then, to show it is positive for  $r \geq a$ it is sufficient
to show it is positive for $r=a$. If we do that, we find $p_1(a)=845.80$ Therefore all conditions
in theorem are verified and the solution we have found is indeed the minimizer.

Now we construct phase 1 minimizers  in the case  case $p=-1, q=4$ The condition $ \Lambda(r)=\eta$ for $0\leq r \leq a$ becomes
\begin{eqnarray}
 \int_0^r\{\frac{1}{24}[(s+r)^6-(s-r)^6]+2s\}w(s) \, ds&+& \nonumber \\
\int_r^a\{\frac{1}{24}[(s+r)^6-(s-r)^6]+2r\}w(s) \, ds&=& \eta r.
\end{eqnarray}

If we assume that the function $w(r)$ is continuous in some interval and we differentiate this last equality twice with respect to $r$, we get that $w(r)$ is a polynomial of third degree in $r$. Since we already know the answer, we set $w(r)=a_3r^3+a_1r$ so that $u(r)=w(r)/r=a_3r^2+a_1$. Putting  this $w(r)$ back into equation (3.13), we find

$$a_1=(3a_3( - a^5 + 1))/(5a^3);\quad   (4a^{10} + 42a^5 - 21)=0;\quad
  \eta =(2a_3( - 2a^5 + 7))/(21a).$$
  To meet the second condition we choose   $a=a_0$ where $a_0$
  is  the positive solution of
  \begin{equation}
  (4a_0^{10} + 42a_0^5 - 21)=0; \quad a_0^5=(5\sqrt(21)-21)/4 \cong 0.4782<1.
   \end{equation}
   For $r>a_0$ we also have $ \Lambda (r)-\Lambda(a_0)=$
  \begin{verbatim}
  ( - (8*a**9*r + 8*a**8*r**2 + 63*a**4*r + 63*a**3*r**2 + 21*a**2*r**3 +
21*a*r**4 - 84)*(a - r)*a3)/(210*a*r).
  \end{verbatim}

Since $a_3$ will be positive, we need the expression
\begin{verbatim}
p_2(r)=8*a**9*r + 8*a**8*r**2 + 63*a**4*r + 63*a**3*r**2 + 21*a**2*r**3 +
21*a*r**4 - 84
\end{verbatim}
to be positive for $r>a_0$. Clearly   $p_2'(r)$ is positive and $p_2(a_0)=4(4a_0^{10} + 42a_0^5 - 21)=0$
in view of  the definition (3.14)  of $a_0$. We conclude that  all conditions of theorem 3.3 are satisfied and
$u(r)= (a_3( - 3a_0^5 + 5a_0^3r^2 + 3))/(5a_0^3)$  is the minimizer, where $a_3>0$ is free. This function is increasing and
$u(0)=3(1-a_0^5)>0.$ Then $u(r)>0$ for $0\leq r \leq a_0$ and it achieves its maximum at $r=a_0$ and $u(a_0)=(a_3(2a_0^5 + 3))/(5a_0^3).$ Moreover,
$$m= 4\pi \int_0^{a_0} r^2 u(r) \, dr= 4\pi a_3/5$$
and then
$$\frac{m}{M}=4 \pi (2a_0^5 + 3))/(a_0^3).$$
We conclude that $u(r)$ is a minimizer provided for given $m$ and $M$,
we have $$\frac{m}{M}\leq 4 \pi (2a_0^5 + 3))/(a_0^3).$$
In that case, we choose $a_3=5m/4 \pi.$

Finally we construct  phase 2 minimizers for $p=-1$ and $q=4$. First we prove a sufficient condition for the minimizers we are going to construct.

\begin{theorem}  Let $u$  be a radially symmetric function  such that
$$ 0<u(r)<M \quad \, \mbox{for}\quad    0 <r<a;\quad  u(r)=M \quad \mbox{for} \quad a \leq  r\leq b \quad \mbox{and}\quad  u=0 \quad \mbox{for} \quad  r\geq b.$$
Suppose that  for some $\eta>0$ we have
$$ \Lambda(r)=\eta \quad   \mbox {for} \quad 0\leq r \leq a;  \quad \Lambda(r)\leq \eta \quad \mbox {for} \quad  a\leq r \leq b \quad \mbox {and}\quad  \Lambda(r)\geq \eta \quad \mbox {for} \quad  r \geq b.$$
Then $u$ is a minimizer.
\end{theorem}
\begin{proof} For any $v(r) $ satisfying $0 \leq v(r) \leq M$ and
$$ \int_0^{\infty} r^2v(r) \, dr= \int_0^{\infty} r^2u(r) \, dr$$
we have
$$ \int_0^{\infty} \Lambda(r) r^2v(r) \, dr -\int_0^{\infty} \Lambda(r)  r^2u(r) \, dr=$$
$$\int_0^a \Lambda(r) r^2v(r) \, dr +\int_a^b \Lambda(r) r^2v(r) \, dr $$
$$+\int_b^{\infty} \Lambda(r) r^2v(r) \, dr   -\int_0^a \Lambda(r)  r^2u(r) \, dr-\int_a^b \Lambda(r)  r^2u(r) \, dr\geq $$
$$\eta  \int_0^a r^2v(r) \, dr  +\eta \int_a^{\infty} r^2v(r) \, dr+$$
 $$\int_a^b r^2v(r) \, dr    -\eta \int_0^a  r^2u(r) \, dr-\eta \int_0^a  r^2u(r) \, dr =0$$
 $$\eta ( \int_0^b r^2u(r) \, dr-\eta  \int_a^b r^2v(r) \, dr) +$$
 $$\int_a^b r^2v(r) \, dr    -\eta \int_0^a  r^2u(r) \, dr -\eta \int_0^a  r^2u(r) \, dr =0$$
$$\int_a^b r^2(\eta -\Lambda(r))(M-v(r))r^2 \, dr \geq 0 $$
and, as a consequence of theorem 3.1, the theorem is proved.
\end{proof}
Before starting the calculation, we describe briefly how the minimizer evolves according to the ratio
$m/M.$ If we drop $4\pi$ in front of the rations, then we will show that given a ratio $k$ between the critical ratios for the phase 1 and phase 3 minimizers, that is,
$$ (2a_0^5 + 3))/a_0^3< k < \frac{1}{3} \left(\frac{2}{3}\right)^{3/5},$$
 where $a_0$ is defined by (3.14), then there are $a,b$ with $ 0<a<a_0<b<b_0$ such that
the minimizer is a polynomial $ a_3r^2+a_1$ for $0\leq r \leq a$ and $u(r)=M$ for $a< r \leq b$ and
$u(r)=0$ for $ r\geq b.$ As the ratio $k$ decreases, $a$ goes to the right, $b$ goes to the left and
as $ k$ tends to $(2a_0^5 + 3))/a_0^3$ , $a$ tends to $a_0$ and $b$ also tends to $a_0$ so that the minimizer tends to be purely the critical function. On the other way around, as  $k$ increases, $a$ goes to the left, $b$ goes to the right  and
as $ k$ tends to $\left(\frac{2}{3}\right)^{3/5} $, $a$ tends to $0$ and $b$  tends to $b_0$  so that the minimizer tends to be purely the constant function.

We start the calculations to verify the conditions in theorem 3.4.  The condition $\Lambda(r)=\eta$ for $0\leq r \leq a$ becomes
\begin{equation}
 \int_0^r\{\frac{1}{24}[(s+r)^6-(s-r)^6]+2s\}w(s) \, ds+
\int_r^a\{\frac{1}{24}[(s+r)^6-(s-r)^6]+2r\}w(s) \, ds+$$
$$M \int_a^b\{\frac{1}{24}[(s+r)^6-(s-r)^6]+2r\}s \, ds = \eta r.
\end{equation}

Moreover, we also have
$$ \Lambda(r)=\frac{1}{r}\int_0^a\{\frac{1}{24}[(s+r)^6-(s-r)^6]+2s\}w(s) \, ds$$
 $$+  M\frac{1}{r}\int_a^r\{\frac{1}{24}[(s+r)^6-(s-r)^6]+2s\}s \, ds+$$
 $$ M\frac{1}{r}\int_r^b\{\frac{1}{24}[(s+r)^6-(s-r)^6]+2r\}s \, ds$$
  $$\mbox{for}\quad   a\leq r \leq b$$
  and
$$\Lambda(r)=\frac{1}{r}\int_0^a\{\frac{1}{24}[(s+r)^6-(s-r)^6]+2s\}w(s) \, ds$$ $$+M\frac{1}{r}\int_a^b\{\frac{1}{24}[(s+r)^6-(s-r)^6]+2s\}s \, ds$$
for $r>b.$

If, as before, we set $w(r)=ru(r)=a_3r^3+a_1r$ and put it in (3.15) we get a fifth degree containing the powers 5, 3 and 1. Setting the coefficients of $r^5$ and $r^3$ equal to zero we get

\begin{verbatim}a1:=(M*(4*a**10 - 25*a**7*b**3 + 21*a**5*b**5 + 21*a**5 - 21*b**5))/(4*a**10 +
42*a**5 - 21);
a3:=(35*M*(a**5*b**3 - a**3*b**5 + a**3 - b**3))/(4*a**10 + 42*a**5 - 21);
dd1:=(M*(20*a**14*b**3 - 56*a**12*b**5 + 56*a**12 + 36*a**10*b**7 +
504*a**10*b**2 - 1190*a**9*b**3 + 252*a**7*b**5 - 252*a**7 + 378*a**5*b**7
 + 5292*a**5*b**2- 2205*a**4*b**3 - 2646*a**2*b**5 +
2646*a**2 - 189*b**7 - 2646*b**2))/(126*(4*a**10 + 42*a**5 - 21)).
\end{verbatim}
where $\eta=dd1.$

Moreover
 $$ \frac{m}{4 \pi} =\frac{1}{4 \pi} \int_{R^3} u(x) \, dx= (a_3 \int_0^ar^4 \,dr + a_1 \int_0^a r^2 \, dr) + M \int_a^b r^2 \,dr=$$

 $$=(7M(a^5b^3 - a^3b^5 + a^3 - b^3))/(4a^{10} + 42a^5 - 21)$$
 and then

 $$\frac{m}{4\pi M}=(7M(a^5b^3 - a^3b^5 + a^3 - b^3))/(4a^{10} + 42a^5 - 21)$$
 Since $ \Lambda(r) $ is continuous we must have $ \Lambda(b)=\eta$ and this implies:
 \begin{verbatim}
 (16*a**8*b**3 - 13*a**7*b**4 - 42*a**6*b**5 - 21*a**5*b +
42*a**4*b**7 - 42*a**4*b**2 + 21*a**3*b**8 + 63*a**3*b**3 + 63*a**2*b**4 +
126*a*b**5- 84*a + 63*b**6 - 42*b)=0.
 \end{verbatim}

 Given $m/M=k,$ the idea is to use these last two equation to find $a$ and $b$ as functions of $k$. Since it seems difficult to analyze such solution, we follow a different strategy. We consider the second equation
\begin{verbatim}
 (16*a**8*b**3 - 13*a**7*b**4 - 42*a**6*b**5 - 21*a**5*b +
42*a**4*b**7 - 42*a**4*b**2 + 21*a**3*b**8 + 63*a**3*b**3 + 63*a**2*b**4 +
 126*a*b**5- 84*a + 63*b**6 - 42*b)=0
 \end{verbatim}
and we parametrize their solutions by a parameter $0\leq t\leq  1$ putting $a=tb.$ Then we obtain   a quadratic equation in $b^5:$
\begin{equation}
d_{10} b^{10} +d_5b^5+d_0=0
\end{equation}
 where

\begin{verbatim}
d10=t**3*(16*t**5 - 13*t**4 - 42*t**3 + 42*t + 21);
d5=21*( - t**5 - 2*t**4 + 3*t**3 + 3*t**2 + 6*t + 3);
d0=42*( - 2*t - 1).
\end{verbatim}
If we plot the polynomial
$$p_2(t)=\frac{d_{10}}{t^3}=16t^5 - 13t^4 - 42t^3 + 42t + 21$$
 for $0\leq t \leq 1$ we see that it is positive (see Plot3).  Since $d_0(t)<0$,
 for $0\leq t \leq 1$, if we put $z(t)=b^5(t)$ we get
$$z(t)=\frac{-d_5+(d_5^2-4d_{10} d_0)^{1/2}}{2d_{10}}$$
or
\begin{equation}
z(t)=\frac{-2d_0}{d_5+(d_5^2-4d_{10} d_0)^{1/2}}
\end{equation}
The second formula for $z(t)$ is more convenient because $d_{10}$ vanishes at $t=0$ and then, at that value of $t$, the first formula gives $0/0.$

If we set $t=0$ in (3.16) or (3.17),  we get $b^5=2/3$ and then $b=b_0$ where $b_0$ is the radius of the critical phase 3 minimizer.  If we set  $t=1$ in (3.16)  we get $4b^{10} + 42b^5 - 21=0$ and this gives $b=a_0$, where $a_0$ given by (3.14)  is the radius of phase 3 minimizers. Therefore, we see that as the parameter $t$ goes from $0$ to $1$, $b(t)$ goes from the bottom of the interval for phase 3 minimizers to the  top of the interval defined by $a_0$ and $a(t)=tb(t)$ goes from zero to $a_0$. In fact, $b(t)$ is a decreasing function of $t$ and $a(t)$ is increasing. To see that  $b(t)$ is decreasing,  we denote by $dz$ the derivative by $z(t)$ given by (3.17)  and  plot it (see Plot4). Since $a^5(t)=t^5 b^5(t)$ to see that $a(t)$ is increasing  we plot $da=5z(t) +tz'(t)$ (see Plot5) and we see that it is positive.

Next we analyse the ratio $k(t)=m/M$ where
$$ k(t)=(7(a^5b^3 - a^3b^5 + a^3 - b^3))/(4a^{10} + 42a^5 - 21)$$
If we set $a=tb$ we get

$$ k(t)= (7(b^5t^4 + b^5t^3 + t^2 + t + 1)(t - 1)b^3)/(4b^{10}t^10 +
42b^5t^5 - 21)$$
 If we try to calculate $k(1)$ we get $0/0$ but after the cancelation of the common factor $t-1$, $k(t)$  becomes

\begin{verbatim}
k(t):= (b**3*(16*b**5*t**9 + 3*b**5*t**8 - 55*b**5*t**7 - 42*b**5*t**6 +
42*b**5*t**5 + 63*b**5*t**4 + 21*b**5*t**3 + 16*t**7 + 3*t**6 - 39*t**5
 - 55*t**4 +63*t**2 + 63*t + 21))/(3*(4*b**5*t**11 + 12*b**5*t**10 +
 32*b**5*t**9 -6*b**5*t**8 - 114*b**5*t**7 - 126*b**5*t**6 - 42*b**5*t**5 +
 16*t**7 + 24*t**6 +24*t**5 + 8*t**4 + 21*t**3 + 63*t**2 + 63*t + 21)).
\end{verbatim}

For $t=0$ we get $k(0)=b_0^3/3$ and for $t=1$ we get
$$k(1)=(a_0^3( - 2a_0^5 - 3))/(30(a_0^5 - 1))=a_0^3/(2a_0^5 + 3).$$
The last equality holds because
$$(a_0^3( - 2a_0^5 - 3))/(30(a_0^5 - 1))-a_0^3/(2a_0^5 + 3)=$$
$$=(a_0^3(4a_0^{10} + 42a_0^5 - 21))/(30(2a_0^{10} + a^5 - 3))=0$$
in view of the definition (3.14)  of $a_0$.

If we plot $dk$  the derivative of $k(t)$ we see it is negative (see Plot6).  We conclude that as $t$ goes from zero to one,
$k(t)$ goes monotonically from $b_0^2/3$ to $a_0^3/(2a_0^5 + 3)$ which are the critical ratios for phase 3 and phase 1 minimizers, respectively. Therefore, if we give a ratio $k_0$ between those values,  there is a unique $t_0$, $0<t_0<1$ such that $k(t_0)=k_0$. With $t_0$ in hands we can calculate $a=a(t_0)$ and $b=b(t_0).$ With that choice, the two equalities involving $a$ and $b$  are satisfied.

Next we work with the inequalities that appear in theorem 3.4.If we define $z=r/b$, since  $a=tb$, $a<r<b$ is equivalent to $t<z<1.$ After some factorization, we have to show that
\begin{verbatim}
tt4:= (4*b**5*t**6 + 4*b**5*t**5 + 4*b**5*t**4 - 21*b**5*t**3
 - 21*b**5*t**2- 21*b**5*t - 21*b**5 + 21*t + 21)>0
 \end{verbatim}
 for $0<t<1$ where
 $b^5=z(t)$ given by (3.17)  If we plot $tt4$ we see it is positive (see Plot7).

 As far the other inequality is concerned if we define $r=b*y$
 we have to verify that for a certain polynomial $p_3(y)$ we have
 $$ p_3(y)= ww4(y^4+y^3) + ww2(y^2+y)+ww0 >0\quad \mbox{for} \quad y \geq 1$$
  where
\begin{verbatim}
 ww4:=8*b**5*t**7 + 12*b**5*t**6 + 12*b**5*t**5 - 38*b**5*t**4
  - 63*b**5*t**3- 63*b**5*t**2 - 63*b**5*t - 21*b**5 + 42*t**2 + 63*t + 21;
ww2:=4*b**5*t**11 + 12*b**5*t**10 + 16*b**5*t**9 - 34*b**5*t**8
 - 142*b**5*t**7 - 54*b**5*t**6 + 105*b**5*t**5 + 147*b**5*t**4 +
 147*b**5*t**3 + 63*b**5*t**2 + 16*t**7 + 24*t**6 + 24*t**5 - 76*t**4
 - 126*t**3 +63*t + 21;
ww0:=2*( - 16*b**5*t**9 - 3*b**5*t**8 + 55*b**5*t**7 + 42*b**5*t**6
 - 42*b**5*t**5 - 63*b**5*t**4 - 21*b**5*t**3 - 16*t**7 - 3*t**6 + 39*t**5
  + 55*t**4 - 63*t**2 - 63*t - 21).
 \end{verbatim}
 Plotting ww4 and ww2 (see Plot8 and 9)  we see that they are positive and this implies $p_3(y)$  increases with $y$. Therefore, it is sufficient to show it is positive for $y=1$. Taking $y=1$ and plotting $p_3(1)$  (see Plot10) we see it is positive.

 We have still to verify that $M$ is indeed the maximum of $u(r).$ To do that we define
\begin{verbatim}
c1=a1/M=(4*a**10 - 25*a**7*b**3 + 21*a**5*b**5 + 21*a**5
 - 21*b**5)/(4*a**10 + 42*a**5 - 21);
c3=a3/M=35*(a**5*b**3 - a**3*b**5 + a**3 - b**3))/(4*a**10 + 42*a**5 - 21)=
=((a**4*b**3 + a**3*b**4 + a**2 + a*b + b**2)*(a - b))/(4*a**10 + 42*a**5 - 21).
\end{verbatim}

Clearly $c_3>0$ because $a<b$ and the denominator is negative. Next we show that $c_1>0$.
Since the denominator of $c_1$ is negative, we have to verify that
\begin{verbatim}
nc1=(4*a**10 - 25*a**7*b**3 + 21*a**5*b**5 + 21*a**5 - 21*b**5)
\end{verbatim}
is negative. Setting $a=tb$ and replacing $b^5$ by $z(t)$ given by (3.17)   and plotting $nc1$, we verify that $ nc1<0$
(see Plot11). We conclude that the maximum of $u(r)$ is assumed at $r=a.$

Defining $tt:=c_3a^2+c_1-1$ we have to verify that
\begin{verbatim}
tt:=( - 24*a**7*b**3 + 20*a**5*b**5 - 20*a**5 - a**2*b**3
 - 21*b**5 + 21)/(4*a**10 +42*a**5 - 21)<0.
\end{verbatim}
 Since the denominator of $tt$ is negative, we need the numerator
\begin{verbatim}
ntt :=- 24*a**7*b**3 + 20*a**5*b**5 - 20*a**5 - a**2*b**3 - 21*b**5 + 21
\end{verbatim}

 to be positive. Plotting $nt$ we see it is indeed positive (see Plot12).

 We conclude that all conditions of theorem 3.4 are satisfied and that the functions we have constructed are indeed phase 2 minimizers.

 \begin{figure}[h]
    \begin{center}
      \includegraphics[width=.4\linewidth]{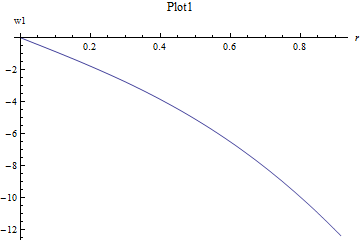}
      \hspace{.15\linewidth}
      \includegraphics[width=.4\linewidth]{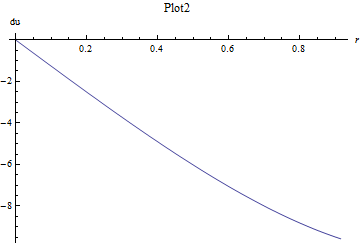}
      \\
      \includegraphics[width=.4\linewidth]{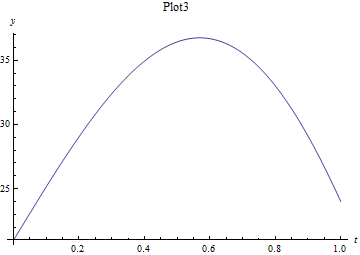}
      \hspace{.15\linewidth}
      \includegraphics[width=.4\linewidth]{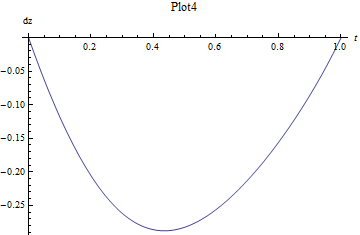}
      \\
      \includegraphics[width=.4\linewidth]{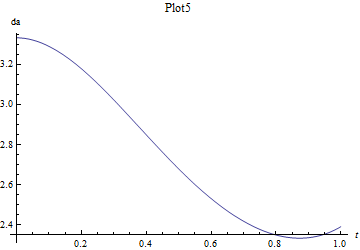}
      \hspace{.15\linewidth}
      \includegraphics[width=.4\linewidth]{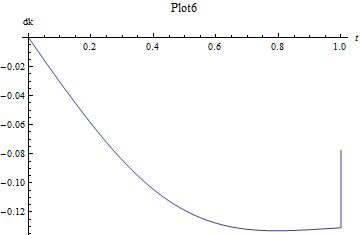}
    \end{center}
    \label{figura1}
    \caption{}
  \end{figure}

  \begin{figure}[h]
    \begin{center}
      \includegraphics[width=.4\linewidth]{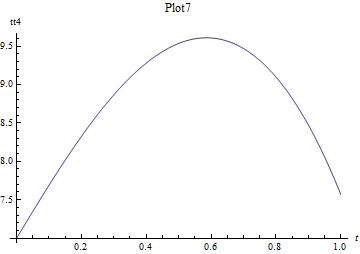}
      \hspace{.15\linewidth}
      \includegraphics[width=.4\linewidth]{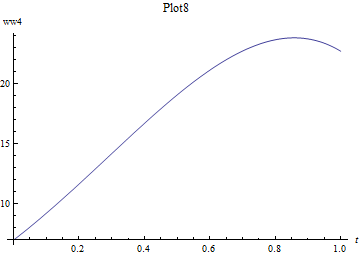}
      \\
      \includegraphics[width=.4\linewidth]{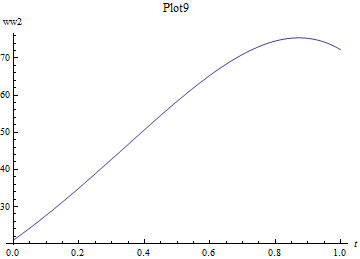}
      \hspace{.15\linewidth}
      \includegraphics[width=.4\linewidth]{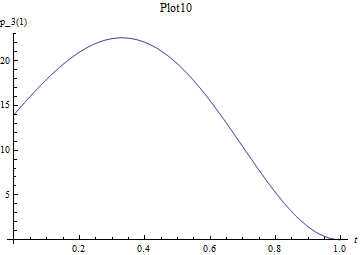}
      \\
      \includegraphics[width=.4\linewidth]{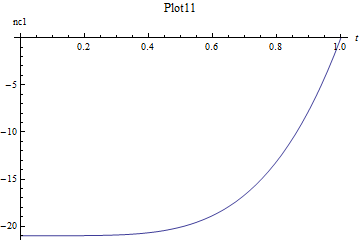}
      \hspace{.15\linewidth}
      \includegraphics[width=.4\linewidth]{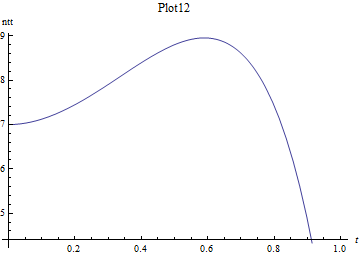}
    \end{center}
    \label{figura2}
    \caption{}
  \end{figure}

\end{document}